\documentclass{amsproc}
\usepackage{amsmath,amssymb,amsfonts,amsthm,url,xspace,graphicx,color}

\numberwithin{equation}{section}

\newtheorem{thm}{Theorem}

\newtheorem{problem}{Problem}
\newcommand{\Z}{\mathbb{Z}}

\newcommand{\R}{\mathbb{R}}
\newcommand{\C}{\mathbb{C}}
\newcommand{\G}{\mathcal{G}}
\newcommand{\M}{\mathcal{M}}
\newcommand{\E}{\mathbb{E}}
\newcommand{\V}{\mathcal{V}}
\newcommand{\SL}{\mathrm{SL}}
\newcommand{\GL}{\mathrm{GL}}
\newcommand{\Tr}{\mathrm{Tr}}

\renewcommand{\Pr}{\mathrm{Pr}}
\newcommand{\be}{\begin{equation}}
\newcommand{\ee}{\end{equation}}
\newcommand{\old}[1]{}

\begin{document}
\title{Some combinatorial problems arising in the dimer model}
\author{R. Kenyon}
\address{Mathematics Department, Yale University, New Haven}
\email{richard.kenyon@yale.edu}

\subjclass[2020]{Primary 82B20}
\date{}
\maketitle

\section{Introduction} 

We collect here a few diverse problems on the dimer model.
For background on dimers one can read for example \cite{Kenyon.lectures}.

Let $\G=(V,E)$ be a finite, connected, bipartite graph.
A \emph{dimer cover}, or \emph{perfect matching}, of $\G$
is a collection $m\subset E$ of edges that covers each vertex exactly once.
We denote by $\M=\M(\G)$ the set of dimer covers of $\G$.

Let $c:E\to\R_{>0}$ be a positive weight function edges of $\G$. If $m\in\M$ is a dimer cover, 
we define $c(m)=\prod_{e\in m}c_e$ to be the \emph{weight} of $m$. 
This weight function defines a probability measure $\mu=\mu_{c}$ on $\M$, giving a dimer cover a probability proportional to its weight,
that is,
$\Pr(m) = \frac{c(m)}{Z}$, where the constant of proportionality $Z$ is the \emph{partition function}: 
$Z=\sum_{m\in\M}c(m).$

Essentially the most fundamental problem in the dimer model is the following:
\begin{problem}\label{p1} For a random dimer cover, compute the edge probabilities, that is, compute 
the probability $\Pr(e)$ that any particular edge $e$ is used. 
\end{problem}

This is easy if we can compute $Z$ as a function of edge weights, since 
\be\label{dZ}\Pr(e) = \frac{c_e}{Z}\frac{\partial Z}{\partial c_e} = \frac{\partial \log Z}{\partial \log c_e}.\ee
However on general graphs computing $Z$ is hard, even for constant edge weights: enumerating dimer covers is in fact $\#$P-complete, by a celebrated result of Valiant \cite{Valiant}. 
Note that computing $\Pr(e)$ is in fact just as hard as computing $Z$, since one can reconstruct $Z$ from integrating $\Pr(e)$ using (\ref{dZ}).
For planar graphs we can however compute $Z$ and thus solve the above problem, 
using Kasteleyn's method of enumerating dimer covers with determinants, see
\cite{Kenyon.lectures} and Section \ref{Kastmtx} below.

One useful change of perspective is to consider the edge weights as a ``$\C^*$-local system" on $\G$,
that is, as a connection on a line bundle on the graph. This point of view may seem at first
just an added level of abstraction, and to be taking us away from combinatorics. But it turns out to have a number of advantages, such as:

\begin{enumerate}
\item It highlights a non-obvious symmetry of the problem: the gauge invariance (see below).
\item It generalizes naturally to other gauge groups such as $\SL_n(\R)$, with interesting combinatorial interpretations (see Sections \ref{SL2scn} and \ref{SLnscn} below).
\item It connects the problem to geometry. 
\end{enumerate}

Our choice of problems in this article is (mostly) motivated by this geometric point of view.

\section*{Acknowledgements}  
	
This research was supported by NSF grant DMS-1940932 and the Simons Foundation grant 327929.

\section{Preliminaries}

\subsection{Vector bundles and connections on graphs}

What is a bundle with connection on a graph?
Fix a vector space $\V$ and assign an isomorphic copy $\V_v$ of $\V$ to each vertex $v\in V$.
For each edge $e=uv$ we associate a linear isomorphism $\phi_{uv}:\V_u\to\V_v$
of these vector spaces, with $\phi_{vu}=\phi_{uv}^{-1}$. 
Generally we have $\phi_{uv}\in\GL(\V)$ but we can restrict if desired to other
subgroups of $\GL(\V)$, and even noninvertible endomorphisms \cite{DKS}. In this paper we will only
use $\SL_n(\C)$ connections on $\C^n$ or $\C^*$-connections on $\C$. 
(Here $\C^*=\C\setminus\{0\}$ is the group of $\C$-linear automorphisms of $\C$.)
These are called, respectively, $\SL_n$-local systems  and $\C^*$-local systems.
Sometimes we specialize the latter from $\C^*$ to $\R^*$. 

If we choose a basis for each vector space $\V_v$ then each $\phi_{uv}$ is a matrix. Changing bases by matrices
$A_v$ at vertices $v$ results in new edge matrices $\phi_{uv}'=A_u\phi_{uv}A_v^{-1}$. Such changes are
called \emph{gauge transformations}. They form a symmetry of the system. 

If $\gamma$ is an oriented loop in $\G$ based at $x_0$, the \emph{monodromy} $\phi_\gamma$ of the connection around $\gamma$
is the composition of the isomorphisms $\phi_e$ along $\gamma$ starting at $x_0$. The conjugacy class of $\phi_\gamma$ is invariant under 
gauge transformations. Note also that changing the basepoint along $\gamma$ conjugates the monodromy. 

For one-dimensional bundles, with $\V=\R$,
each $\phi_{uv}\in\R^*$  is a scalar. Since we are assuming $\G$ is bipartite, we can canonically orient the edges from black to white.
Then $\phi_{bw}$ is a scalar quantity associated to edge $bw$ which we call the \emph{edge weight}. 
In probability settings we usually take $\phi_{bw}$ positive.
However it is sometimes useful to consider $\V=\C$ and $\phi_{bw}\in\C^*$.
Gauge transformations are functions $A:V\to\C^*$ and transform $\phi_{bw}$ into $\phi_{bw}' = A_b\phi_{bw}A_w^{-1}$. 

For an $\R$-bundle, when edge weights $\phi_{bw}$ are positive, we can associate a probability measure $\mu$ as above. 
Then positive gauge transformations (i.e. when the $A_v>0$) change the weights of individual dimer covers
but not the probability measure $\mu$:
If we multiply the edge weights of all edges at a single vertex by a constant $\lambda>0$, then the weight of every
dimer cover is multiplied by $\lambda$. This implies that the probability measure $\mu$ only depends on the 
equivalence class of edge weights under gauge transformations, that is, on the ``edge weights modulo gauge".

\subsection{Planar bipartite graphs and the Kasteleyn matrix}\label{Kastmtx}

Let $\G$ be a planar bipartite graph with a $\C^*$-local system $\Phi$ (equivalently, a set of edge weights in $\C^*$, see the previous section).
We define a matrix $K$, the \emph{Kasteleyn matrix for $\G$}, as follows. 
The matrix $K=(K_{wb})$ has rows indexing white vertices and columns indexing black 
vertices, and 
$$K_{wb} = \begin{cases} \pm \phi_{bw}&b\sim w\\0&\text{otherwise}\end{cases}$$
where the signs are chosen by the ``Kasteleyn rule": a face of length $\ell$ has $\frac{\ell}2+1\bmod 2$ minus signs.
Kasteleyn showed in this case that the determinant of $K$ counts weighted dimer covers:

\begin{thm}[\cite{Kast, Kuperberg}]
$|\det K| = \sum_{m\in\M}c(m).$
\end{thm}

As an application, probabilities of one or more edges being in a dimer cover are minors of the inverse of $K$, see \cite{Kenyon.localstats}.

If $\G$ has an $\SL_n$-local system, one can similarly define a Kasteleyn matrix, whose entries are now $n\times n$ matrices themselves;
see Sections \ref{KastSL2},\ref{KastSLn} below.

\section{From edge weights to probabilities}

We wish to study the mapping from edge weights to probabilities.
This requires first defining the correct spaces. 

\subsection{Fractional matchings}

A \emph{fractional matching} of $\G=(V,E)$ is a function $f:E\to[0,1]$ which sums to one at each vertex:
$\sum_{u\sim v}f(u) = 1$. A dimer cover $m$ determines a fractional matching, by assigning $f(e)=1$ for each $e\in m$
and $f(e)=0$ otherwise. The set of fractional matchings $\Omega(\G)$ is a convex polytope in $\R^E$ (in fact a subset of $[0,1]^E$)
whose vertices are exactly the dimer covers of $\G$, see \cite{LP}. 

A probability measure on dimer covers is a probability measure on the vertices of $\Omega$, and so has a barycenter which is a point in
$\Omega$. The coordinate of this barycenter corresponding to an edge is the probability of that edge.

\old{

}

So a more elegant statement of Problem \ref{p1} is
\begin{problem}Compute the expected fractional matching.
\end{problem}

See an example in Figure \ref{2by3}.
\begin{figure}
\begin{center}\includegraphics[width=1.5in]{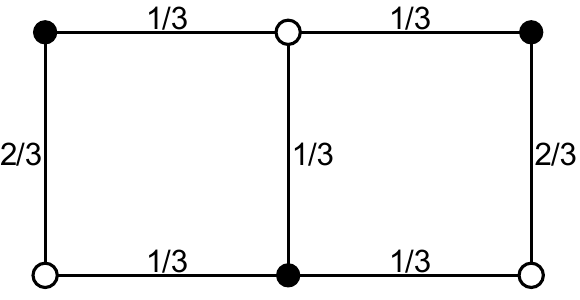}\end{center}
\caption{\label{2by3}A graph and its expected fractional matching (for uniform edge weights).
This graph has three dimer covers; two of these use the leftmost vertical edge, so its probability is $2/3$.}
\end{figure}

\subsection{Cycles}

A \emph{flow} on a graph $\G$ is a function $f$ on oriented edges satisfying $f(-e)=-f(e)$ where $-e$ represents the edge with the reversed orientation. The \emph{cycle space} $H_1(\G)$ is the space of divergence-free flows: flows $f$ satisfying $\sum_{u\sim v}f(vu)=0$ for all $v$. 

Since $\G$ is bipartite, we can orient edges from black to white, and then for any two fractional matchings $f_1,f_2\in\Omega$,
$f_1-f_2$ represents an element of the cycle space of $\G$, denoted $[f_1-f_2]$.
Fixing a basepoint $f_0\in\Omega$, we can map $\Omega$ linearly and injectively into $H_1(\G)$ by $f\mapsto[f-f_0]$.

We say $\G$ is \emph{nondegenerate} if $\Omega$ is of dimension $d=|E|-|V|+1$, the dimension of the cycle space of the graph.
A graph is degenerate (that is, not nondegenerate) if there are unused edges (edges which participate in no dimer cover) whose removal does not disconnect the graph \cite{LP}.
Figure \ref{degen} shows a degenerate graph.
\begin{figure}
\begin{center}\includegraphics[width=1.3in]{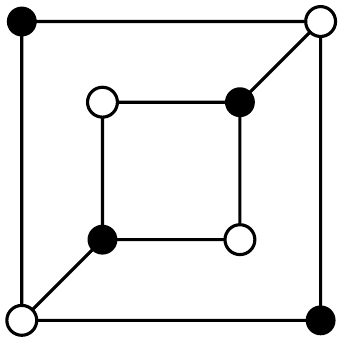}\end{center}
\label{degen}
\caption{A degenerate graph. The dimension of the cycle space is $3$ but $\Omega$ is of dimension $2$. Note that the two diagonal edges are ``unused", that is, do not appear in any dimer cover.}
\end{figure}

\subsection{Gauge equivalence}

For a bipartite graph, the space of equivalence classes of edge weights modulo gauge transformations,
that is, the space of $\C^*$-local systems,
is isomorphic to $(\C^*)^d$ where $d$ is the dimension of the cycle space of the graph. 
Explicitly, for a cycle $\gamma$ we can define $X_\gamma$ to be the monodromy of the connection around $\gamma$. In terms of edge weights, $X_\gamma$ is the alternating product of edge weights around $\gamma$: 
assuming we start at a black vertex, $X_{\gamma}$ is the first weight, divided by the second,
times the third, and so on. 
Note that $X_\gamma$ is invariant under gauge transformations.
Reversing the orientation inverts the monodromy.
If $\gamma$ ranges over a basis for the cycle space, the set of monodromies $\{X_\gamma\}$ parameterize the space of 
all edge weights modulo gauge, that is, all $\C^*$-local systems.

For a planar graph, a basis for the cycle space is the set of cycles around the bounded faces.
Consequently we can take the ``face weights" $\{X_f\}_{f\in F}$ on the bounded faces of the graph to parameterize
the choices of edge weights modulo gauge. 

Now a perhaps more interesting variant of Problem \ref{p1} is:

\begin{problem} For nondegenerate graphs, study the map $\Psi$ from cycle weights $\{X_f\}$ to the expected fractional matching.
\end{problem}

The domain and range of $\Psi$ have the same dimension: the dimension of the cycle space of the graph.

\begin{thm} If $\G$ is nondegenerate then $\Psi$ is a diffeomorphism from the space of cycle weights $(\R_{>0})^d$ to $\Omega(\G)$. 
\end{thm}

\begin{proof}
Fix the gauge by choosing a spanning tree $T$ of $\G$ and setting edge weights $1$ on the tree $T$ and weights
$c_e$ for $e\not\in T$.  Let $Z=Z(\{c_e\})$ be the associated partition function. Let $p_e$ be the probability of edge $e$ for $e\not\in T$.
All the other edge probabilities (and therefore the fractional matching $f$) are uniquely determined by the probabilities $\{p_e~|~e\not\in T\}$: 
one can determine the value of $f$ on the leaves of the tree, then remove these and continue with the leaves of the subtrees inductively.
As in (\ref{dZ}) we have
$p_e = c_e\frac{\partial\log Z}{\partial c_e}.$
We can write this as $\vec p = \nabla \log Z$ where the gradient is taken with respect to the vector of ``edge energies" $(\log c_e)_{e\not\in T}$.

Since $Z$ is a positive polynomial (that is, its coefficients are nonnegative), it is, up to scale,
the probability generating function of a probability measure on $\Z^{d}$ of finite support.
The Hessian of $\log Z$ with respect to the logs of its variables is the covariance matrix of this probability measure:
$$c_1c_2\frac{\partial^2 }{\partial c_1\partial c_2}\log Z=\frac{c_1c_2}Z\frac{\partial^2 Z}{\partial c_1\partial c_2} - \frac{c_1c_2Z_{c_1}Z_{c_2}}{Z^2} = \E[X_1X_2]-\E[X_1]\E[X_2]$$
where $X_1,X_2$ are the indicator functions of edges $1,2$. 
 Any covariance matrix is positive semidefinite, and thus this Hessian is positive semidefinite: 
$\mathrm{Hess} \log Z\ge 0$. In fact it is positive definite as long as the Newton polytope of $Z$ is of positive volume,
that is, if the exponents of the monomials of $Z$ span the cycle space. But this is guaranteed by nondegeneracy. 
Since $\mathrm{Hess}(\log Z)$ is positive definite, $\log Z$ is strictly convex and so
$\nabla \log Z$ is a diffeomorphism onto its image.
\end{proof}

Note that $\det(\nabla\Psi)$ is a rational function of the $c_e$. 
What is the degree of $\nabla\Psi$ as a rational map on $\C^d$?
If we extend $\Psi$ to all of $\C^d$ it appears that the inverse of a point $\vec p\in\Omega$ consists in \emph{real} values of $c_e$
(only one preimage consists in positive reals, but there can be multiple non-positive preimages; it is not clear that these are always real,
although they are in the small examples we tried). A similar situation, for resistor networks, where reality of the other preimages is proved, can be found 
in \cite{AK}.

\section{Dimer walks}

To each dimer cover $m\in\M(\G)$ is associated an involution $\pi_m:V\to V$. This is the involution
exchanging a vertex with the vertex it is paired with. 

\begin{problem} Let $\{m_1,m_2,\dots\}$ be a sequence of i.i.d.\! dimer covers of $\G$. Study the random walks on $S_V$, the permutation group on the vertices, defined by 
the $\pi_{m_i}$. 
\end{problem}

Another way to state this is to define $g=\frac1{Z}\sum_{m\in\Omega_1(\G)}\pi(m)\in \R[S_V]$, the element of the group algebra
$\R[S_V]$ of $S_V$ defined by a random choice of dimer cover. What can be said about $g^n$?

As a simple example one can take $\G=K_4$ with vertices $[4]=\{1,2,3,4\}$. The three dimer covers
correspond to the three permutations $(12)(34),(13)(24),(14)(23)$ which generate a subgroup of $S_4$ isomorphic to $(\Z_2)^2$. The dimer random walk just becomes simple random walk on this group, 
where every element
is adjacent to any other element, that is, we have simple random walk on $K_4$.

For another simple example, let $\G$ be the $3\times 2$ grid. Rather than record the element in $S_6$, without loss of information 
we can record simply
the permutations of the $x$-coordinates of the vertices, in $S_3$. There are $3$ dimer covers,
and the walk corresponds to multiplication by $g=\frac13(1+e_{(12)}+e_{(23)})$. Its eigenvalues
are $1,\frac23,\frac23,-\frac13,0,0$.  

As a more interesting example, when $\G$ is the $n\times n$ grid graph on a torus, 
each vertex undergoes a simple random walk on $V$,
since with probability $1/4$ it is matched to any neighbor. These simple random walks are coupled to avoid each other.
How quickly does the process mix? What is the expected winding of two walks around each other, as a function of time?

\section{Magnetic double dimer model}

\subsection{Double dimers}
Recall that $\M=\M(\G)$ is the set of dimer covers of $\G$.
A \emph{double dimer cover} is a function $m:E\to\{0,1,2\}$ which sums to $2$ at each vertex. See Figure \ref{doubledimer}.
\begin{figure}
\begin{center}\includegraphics[width=3in]{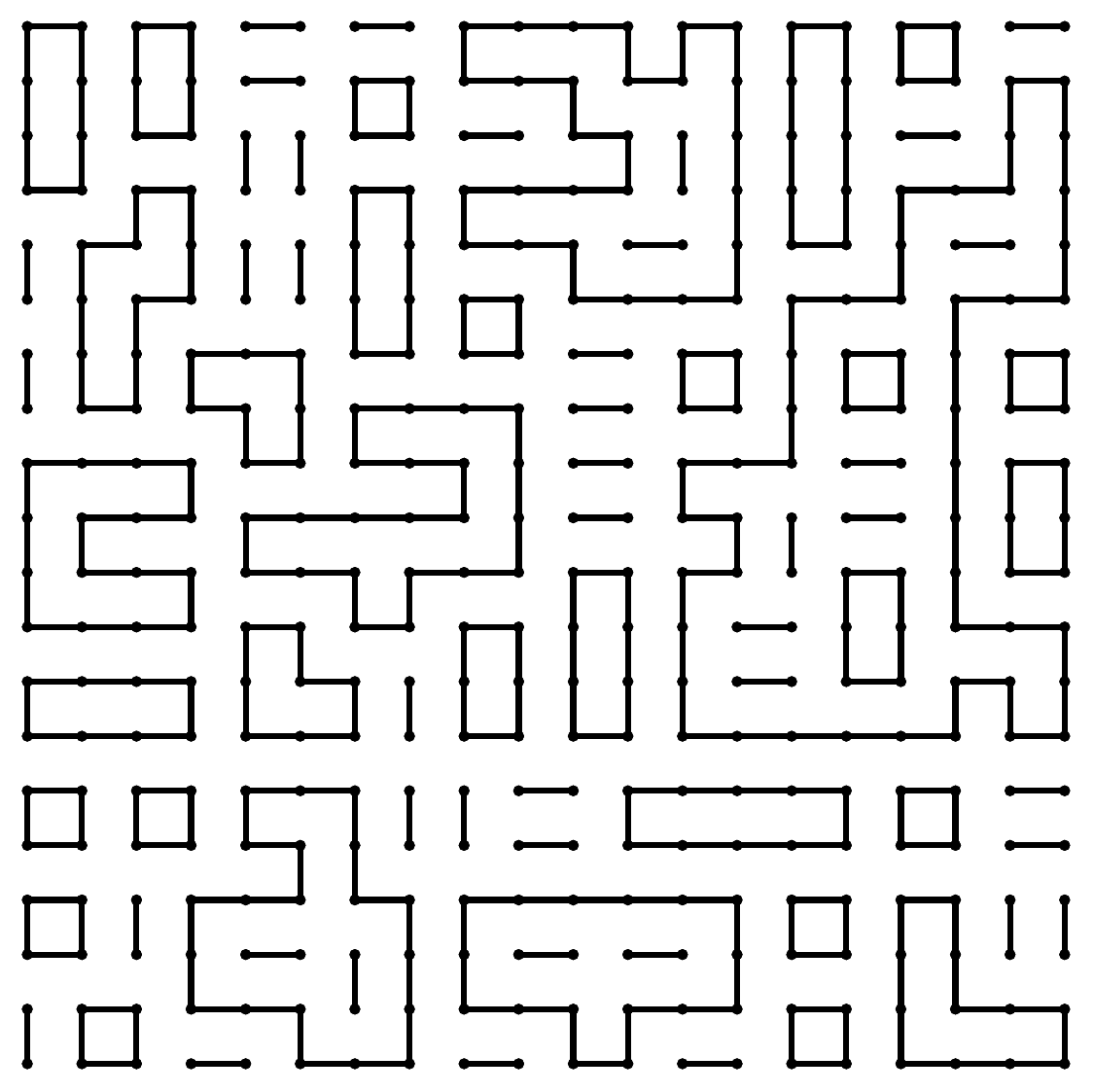}\end{center}
\caption{\label{doubledimer} A double dimer cover of a grid. Edges of multiplicity $1$ or $2$ are shown.}
\end{figure}
A union of two single dimer covers determines a double dimer cover, and (for bipartite graphs) 
every double dimer cover occurs this way.
The map from pairs of dimer covers to double dimer covers is not, however, injective: 
each double dimer cover $m$ arises from $2^\ell$ ordered pairs of single dimer covers,
where $\ell$ is the number of cycles in $m$, where the cycles are formed from the edges of multiplicity $1$.

Let $\M_2$ be the set of double dimer covers of $\G$. 
The natural probability measure on $\M_2$ (for us) is \emph{not} the uniform measure. It is, rather, the projection of the
uniform measure on $\M\times\M$ under the standard map $\M\times\M\to\M_2$. Thus a double dimer cover $m$
has a probability proportional to $2^\ell$ where $\ell$ is the number of cycles (cycles formed from edges of multiplicity $1$; edges of multiplicity $2$ do not count as cycles).
We call $2^\ell$ the \emph{weight} $c(m)$ of $m$.

\begin{problem}\label{p4}
What can be said about the distribution of loops in the double dimer model on $\Z^2$?
\end{problem}

Let $\G$ be a bipartite planar graph, and $q\in\C^*$ be a variable. 
One can define a $\C^*$-connection on $\G$ so that each face has counterclockwise monodromy $q$, see for example
Figure \ref{qweights}.
\begin{figure}
\begin{center}\includegraphics[width=2in]{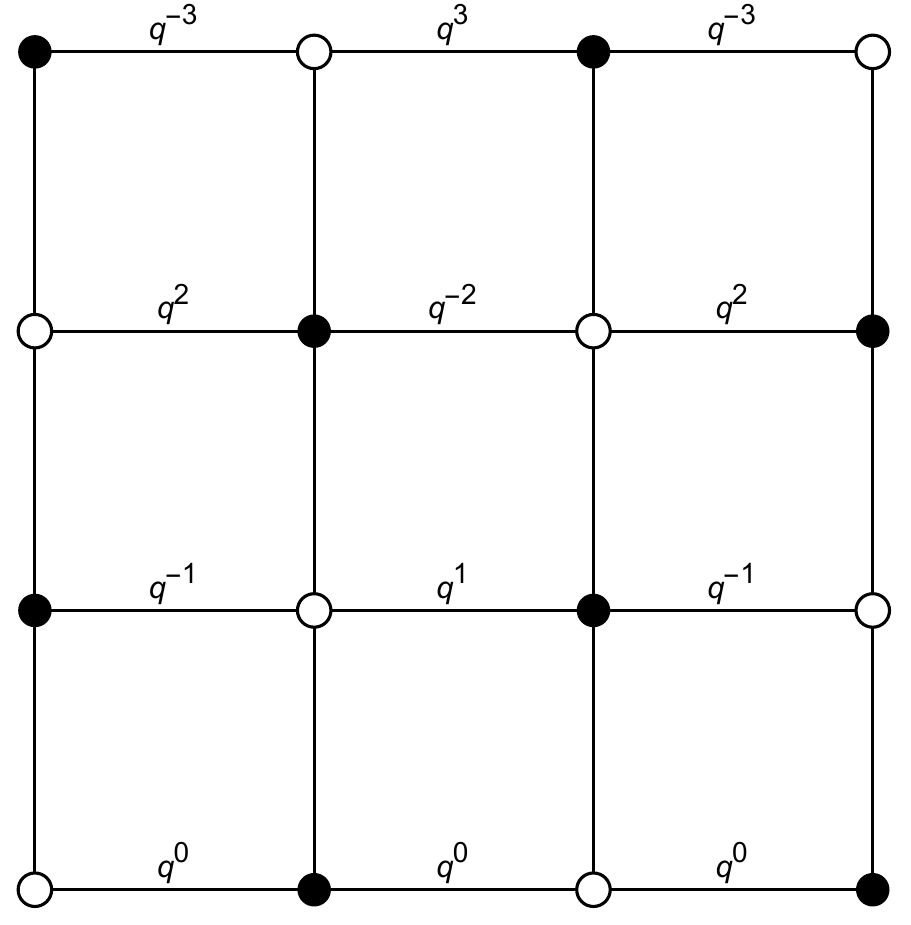}\hskip1cm\includegraphics[width=2in]{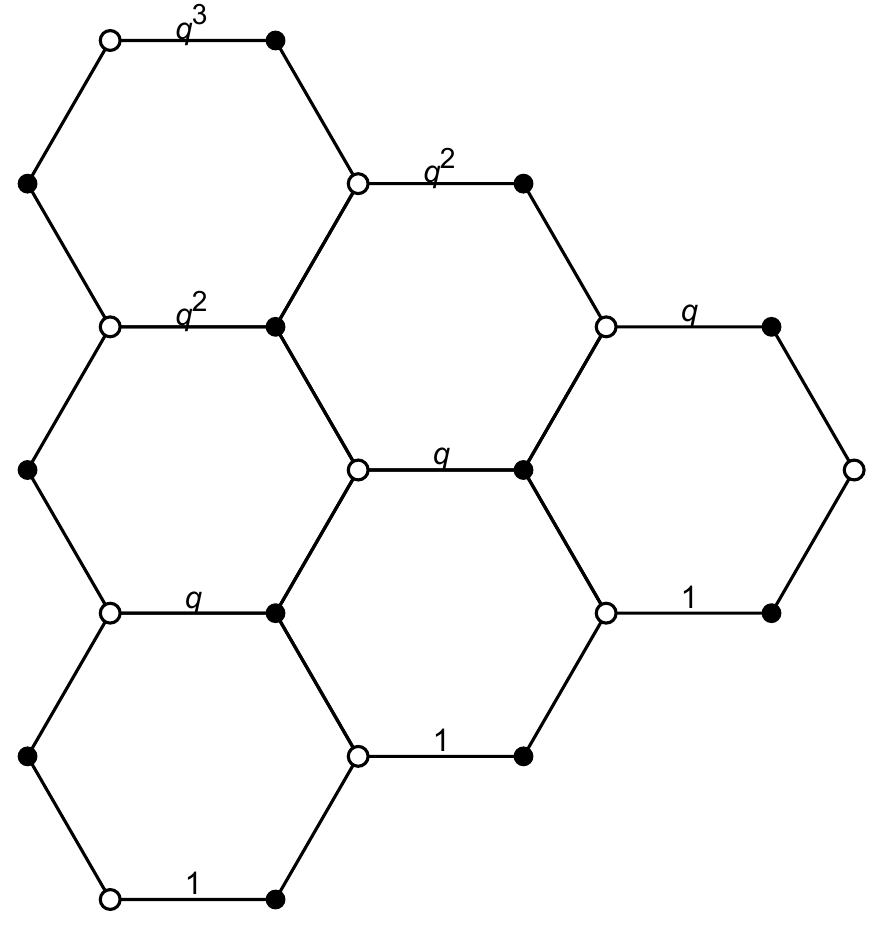}\end{center}
\caption{\label{qweights}Connections with monodromy $q$ per face.}
\end{figure}
Let $K(q)$ be an associated Kasteleyn matrix.
For a double-dimer configuration $m_2$ define
\be\label{cm2}c(m_2) = \prod_{\text{loops $\gamma$ of $m_2$}}(q^{A_\gamma}+q^{-A_\gamma})\ee
where $A_\gamma$ is the area (number of faces) enclosed by $\gamma$.
We then have
\begin{thm} 
\be\label{cm2det} \det K(q) \det K(1/q)=\sum_{m_2\in\M_2}\,\, \prod_{\text{loops $\gamma$ of $m_2$}}(q^{A_\gamma}+q^{-A_\gamma}).\ee
\end{thm}
\begin{proof}
\old{
The quantity $\det K(q)$ counts
single-dimer configurations $m$ with weight 
$c(m) = Cq^h$
where $C$ is a constant and $h$ is the total height function: the sum of the values of the height on all faces (see \cite{Kenyon.lectures} for the definition of heights).
This follows, for example, because the set of dimer covers is connected by face moves (see Figure \ref{facemove} and \cite{Kuperberg}) 
each of which changes
the height by $\pm 1$ and the weight by $q^{\pm1}$. 
\begin{figure}
\begin{center}\includegraphics[width=2.5in]{facemove}\end{center}
\caption{\label{facemove}If every other edge of a face is in a dimer cover, one can rotate the edges by one notch around the face
to create another dimer cover. By a theorem of Kuperberg \cite{Kuperberg}, 
these ``face moves" connect the set of all dimer covers of any planar bipartite graph.}
\end{figure}
Likewise $\det K(1/q)$ counts single-dimer covers with weight $c(m) = Cq^{-h}.$
}

The quantity $\det K(q)\det K(1/q)$ counts pairs of dimer configurations $(m,m')$, where $m'$ has the inverted $q$-weight. In the superposition of $m$ and $m'$, we orient $m$'s dimers from black to white and $m'$'s
dimers from white to black, reversing the parallel transport, so that the orientations of loops are consistent. In this way the weight of an oriented loop is given by its monodromy,
that is, $q^A$ where $A$ is the area enclosed by the loop. When we sum over both orientations of each loop, 
we find (\ref{cm2det}).
\end{proof}

It is tempting to use (\ref{cm2det}) to answer Problem \ref{p4}. For example one can expand $\det K(q)\det K(1/q)$ near $q=1$ to get various moments. However extracting information about the areas from this does not appear easy.

An alternative approach is to invert the standard Kasteleyn matrix (when $q=1$) and use it to compute the probability of any given shape of loop.
For the square grid $\Z^2$, this can be used to compute the expected density of loops of fixed small area, as follows.
The ``inverse Kasteleyn matrix" for $\Z^2$ can be obtained as a limit as $n\to\infty$ of $K^{-1}_{\G_n}$ for $\G_n$ the $2n\times 2n$ grid,
see \cite{Kenyon.lectures}. There is an explicit formula: when $b-w=(x,y)\in\Z^2$,
$$K^{-1}((0,0),(x,y)) = \frac1{(2\pi i)^2}\iint_{|z|=|w|=1}\frac{z^{(-x+y+1)/2}w^{(-x-y+1)/2}}{1+z+w-zw}
\frac{dz}{z}\frac{dw}{w}.$$
Minors of this matrix compute edge probabilities for random dimer covers of $\Z^2$, see \cite{Kenyon.localstats}:
$$Pr((w_1,b_1),\dots,(w_k,b_k)) =|\det(K^{-1}(b_j,w_i)_{1\le i,j\le n})|.$$
The
expected density of loops of area $1$ (the probability that a given face is in such a loop) is the probability that for a pair of single dimer models
$(m,m')$
on $\Z^2$, $m$ contains edges $(0,0)(1,0)$ and $(0,1)(1,1)$ and $m'$
contains edges $(0,0)(0,1)$ and $(1,0)(1,1)$, or the reverse.  The probability that $m$ contains both edges $(0,0)(1,0)$ and $(0,1)(1,1)$ is, by a short computation, $\frac1{8}$.
So the expected density of loops of area $1$ in the double dimer model is $\frac{2}{8^2}=\frac1{32}$.
Similarly the density of loops 
of area $2$ can be computed to be $\frac{(\pi-1)^2}{2\pi^4}.$
The density for area-$3$ loops is already more complicated:
$$\frac{3 \left(64-192 \pi ^2+192 \pi^3-32 \pi ^4-32 \pi ^5+24 \pi^6-8 \pi ^7+\pi ^8\right)}{32 \pi^8}.$$

\section{ $\text{SL}_2$ connections}\label{SL2scn}

\subsection{Kasteleyn matrix}\label{KastSL2}
Consider $\G$ a planar bipartite graph with an $\SL_2$-local system $\Phi$.
Let $K=(K_{wb})$ be an associated Kasteleyn matrix for $\G$, that is a matrix with rows indexing white vertices and columns indexing black 
vertices, and 
$$K_{wb} = \begin{cases} \pm \phi_{bw}&b\sim w\\0&\text{otherwise}\end{cases}$$
where the signs are chosen by the same Kasteleyn rule as in the scalar case: a face of length $\ell$ has $\frac{\ell}2+1\bmod 2$ minus signs.
Here by $0$ we mean the zero matrix in $M_2(\R)$. 
Then $K$ is an $N\times N$ matrix with entries in $M_2(\R)$.
Let $\tilde K$ be the $2N\times 2N$ matrix obtained from $K$ by replacing each entry with its $2\times 2$ array of real numbers.

By a theorem of \cite{DKS}, 
\be\label{SL2sum}\det\tilde K(\Phi)= \sum_{m\in\M_2} \prod_{\text{loops in $m$}}\Tr(\phi_\gamma).\ee
Here $\phi_\gamma$ is the monodromy of the connection $\Phi$ around the loop $\gamma$, and $\Tr(\phi_\gamma)$ is its trace. 
Even though
$\phi_{\gamma}$ depends on a starting vertex and the orientation of the loop, the trace $\Tr(\phi_\gamma)$
is independent of starting vertex (since the trace of a matrix only depends on its conjugacy class) and orientation
(since for matrices $A\in\text{SL}_2$ we have $\Tr(A)=\Tr(A^{-1})$).

\subsection{Flat connections and simple laminations}
Suppose graph $\G$ is drawn on a multiply-connected planar domain $\Sigma$, and $\Phi$ is a \emph{flat} $\SL_2$ connection.
This means the monodromy of $\Phi$ around any cycle in $\G$ which is contractible as a cycle in $\Sigma$, is the identity.
In this case the trace of the monodromy around a loop $\gamma$ only depends on its isotopy class as a loop in $\Sigma$.
In other words two loops in $\G$ which are isotopic as loops in $\Sigma$, have monodromies with the same trace. 

A \emph{simple lamination} is an isotopy class of finite collections of pairwise disjoint simple closed loops on $\Sigma$. 
Let $\Lambda_2$ be the collection of simple laminations.

When $\Phi$ is flat we can group the terms in the sum (\ref{SL2sum}) according to their isotopy classes:

\be\label{SL2sumiso}\det\tilde K(\Phi)= \sum_{\lambda\in\Lambda_2} C_\lambda\Tr(\phi_\lambda)\ee
with $C_\lambda\in\Z_{\ge0}$.

By a Theorem of Fock and Goncharov \cite{FG}, the functions $\Tr(\phi_\lambda)$ as $\lambda$ runs over $\Lambda_2$,
considered as functions of $\Phi$,
are linearly independent and in fact form a linear basis for the regular (i.e. polynomial) functions on the character
variety (the space of flat $\mathrm{SL}_2$-local systems $\Phi$ modulo gauge). 
As a consequence $C_\lambda$ can be in principle determined from $\det\tilde K(\Phi)$ as $\Phi$ varies over flat connections. 
Mysteriously, even though for a finite graph $\det \tilde K(\Phi)$ is a polynomial function of the matrix entries,
and the sum (\ref{SL2sumiso}) is a finite sum, it is not at all clear how to extract the individual $C_\lambda$ from it. 

\begin{problem}\label{P5} How can one extract $C_\lambda$ from $\det\tilde K(\Phi)$ as in (\ref{SL2sumiso})?
\end{problem}

For an example where we know how to extract these coefficients, 
consider the simplest nontrivial case when $\Sigma$ is an annulus. 
Suppose graph $\G$ is drawn on an annulus with flat connection having monodromy $A$ around the nontrivial isotopy class $\gamma$,
the generator of $\pi_1(\Sigma)$.
Then a simple lamination on $\G$ just consists in some number $n$ of copies of loops each isotopic to $\gamma$
(we are ignoring orientation). 
The isotopy classes $\Lambda_2$ are thus in correspondence with the set of nonnegative integers.
We can write 
$$\det\tilde K(A) =\sum_{j=0}^\infty C_j\Tr(A)^j.$$
From this expression it is easy to extract $C_j$: let $\Tr(A)=z$, then the $C_j$ are power series coefficients of $\det\tilde K(A)$.

The next simplest planar surface is a disk with two small disks $p_1,p_2$, removed (also known as a ``pair of pants").
In this case the space $\Lambda_2$ has a simple parameterization by $(\Z_+)^3$
where $(i,j,k)$ corresponds to the lamination with $i,j,$ and $k$ loops surrounding respectively $p_1$ only, $p_2$ only or both $p_1$ and $p_2$,
see Figure \ref{ijk}. 
\begin{figure}
\begin{center}\includegraphics[width=2.5in]{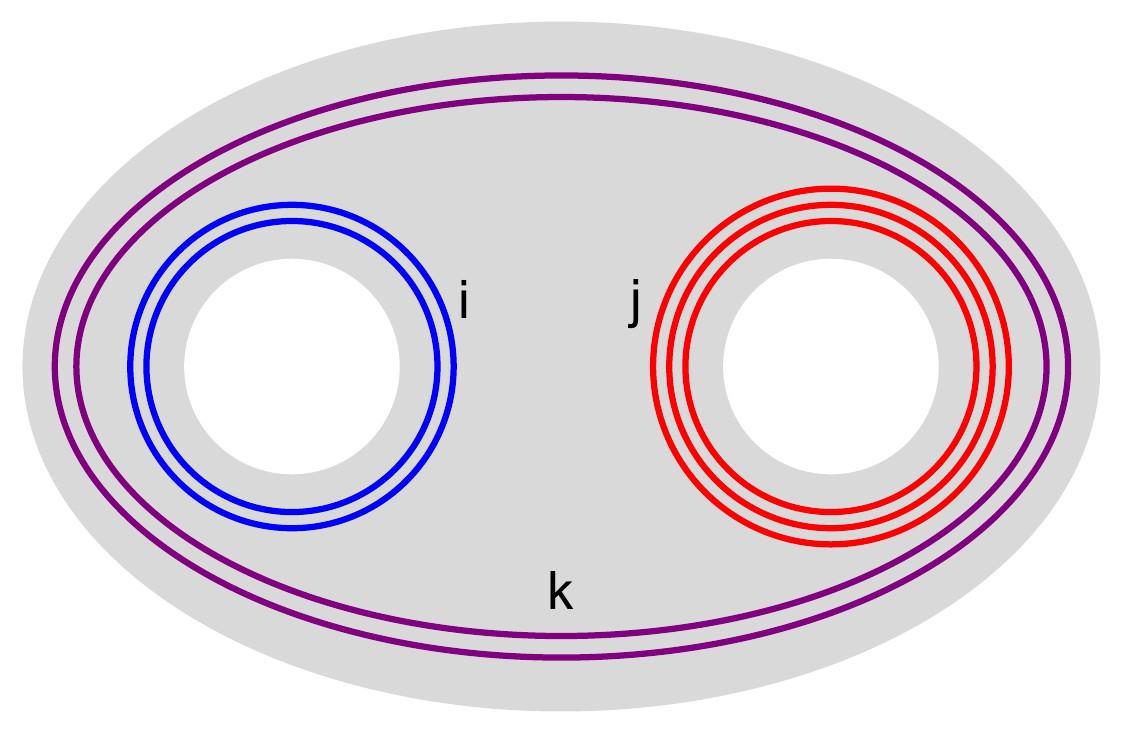}\hskip1cm\end{center}
\caption{\label{ijk}Simple laminations on a pair of pants consist in three kinds of loops:
those surrounding $p_1,p_2$ or both.}
\end{figure}
If $A,B$ are the monodromies around $p_1,p_2$ respectively
then given any $x,y,z\in\C^*$ we can choose $A,B$ in $\mathrm{SL}_2(\C)$ so that $x=\Tr(A), y=\Tr(B), z=\Tr(AB)$.
Then one can extract $C_{i,j,k}$ by a contour integral
$$C_{i,j,k} = \frac1{(2\pi i)^3}\int_{(S^1)^3} \frac{\det K(\Phi)}{x^iy^jz^j}\frac{dx}{x}\frac{dy}{y}\frac{dz}{z}.$$

For a general punctured surface Fock and Goncharov \cite{FG} describe a different (and orthogonal) basis on the space of laminations, the ``Peter-Weyl basis",
and show that the basis $\{\Tr(\phi_\lambda)\}_{\lambda\in\Lambda_2}$ is obtained from the Peter-Weyl basis by an upper triangular transformation. 
One route to solving Problem \ref{P5} would be to write down this base-change matrix.

\section{$n$-fold dimer model}\label{SLnscn}

\subsection{$n$-multiwebs}
An \emph{$n$-multiweb}, or \emph{$n$-fold dimer cover} in $\G$ is a function $m:E\to\{0,1,2,\dots,n\}$ which sums to $n$ at each vertex. 
So a $1$-multiweb is a dimer cover and a $2$-multiweb is a double dimer cover.
A union of $n$ single dimer covers is an $n$-multiweb, and every $n$-multiweb occurs this way
(although not uniquely).

Let $\M_n$ be the set of $n$-multiwebs of $\G$. 
The natural probability measure on $\M_n$ is, like in the double dimer case, the projection of the
uniform measure on $\M^n$ under the standard map $\M^n\to\M_n$. 
Unlike the $n=2$ case, the size of the preimage is difficult to compute in general, in fact $\#$P-complete:
if $n=3$ and a $3$-multiweb is a trivalent graph (that is, all edges have multiplicity $1$),
then the triples of dimer covers in the preimage are in bijection with edge $3$-colorings, or Tait colorings.
Tait colorings are colorings of the edges with three colors so that each color appears at each vertex.
For planar graphs Tait colorings are dual to proper $4$-colorings of the vertices of the dual graph.

\subsection{Kasteleyn matrix}\label{KastSLn}
Consider $\G$ a planar bipartite graph with an $\SL_n$-local system $\Phi$.
Let $K=(K_{wb})$ be a Kasteleyn matrix for $\G$; $K$ is an $N\times N$ matrix with entries in $M_n(\R)$.
Let $\tilde K$ be the $nN\times nN$ matrix obtained from $K$ by replacing each entry with its $n\times n$ array of reals.

By a theorem of \cite{DKS}, 
\be\label{SLnsum}\det\tilde K(\Phi)= \sum_{m\in\M_n} \Tr(\phi_m)\ee
where we still need to define the trace of an $n$-multiweb with an $\SL_n$-connection; see the next section.

\subsection{Trace of a $3$-multiweb}

The trace of an $n$-multiweb $m$ is not simple to define; it is a contraction of $n$-tensors defined at each vertex. We refer to \cite{DKS} for the general definition and give here a working 
definition for $n=3$. We cut each edge of $\G$ into two half-edges, one associated with each of its vertices.
We consider colorings $c$ of the half-edges of $m$ with three colors $\{1,2,3\}$, so that an edge of multiplicity $j\in\{0,1,2,3\}$ gets a set of $j$ colors,
and so that all colors appear at each vertex of $\G$. See Figure \ref{halfcolorings}. 
\begin{figure}
\begin{center}\includegraphics[width=2.5in]{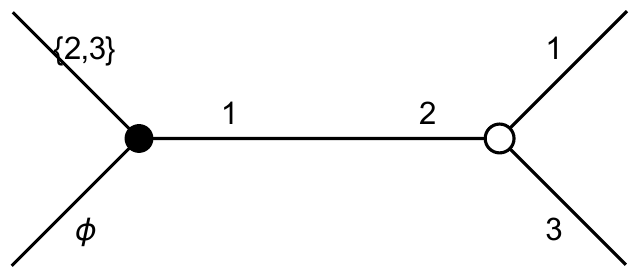}\end{center}
\caption{\label{halfcolorings}$3$-coloring of the half-edges of $\G$.}
\end{figure}

Then on an edge $e$ of multiplicity $j$ we have two sets of colors $S,T$ of size $j$,
with $S$ located near the white vertex and $T$ near the black vertex. We define 
\be\label{tracedef}Tr(c) = \sum_{\text{colorings $c$}}sgn(c)\prod_e(\phi_{bw})_S^T\ee
where $(\phi)_S^T$ denotes the $S,T$ minor of $\phi$.
Here $sgn(c)$ is a sign depending on the cyclic order of colors at each vertex: at a black vertex if the colors are in counterclockwise order we get sign $+$
and otherwise sign $-$, and this convention is reversed at a white vertex; the product of these signs over all vertices is $sgn(c)$.
(For edges with multiplicity the colors are assumed in natural order within that set.)

When $\Phi$ is the identity connection, and $m$ is planar, $|\Tr(m)|$ is the number of
Tait colorings of $m$ (see \cite{DKS}). In other words, all nonzero terms of (\ref{tracedef}) have the same sign. The global sign of $\Tr(m)$ is unimportant since it depends on an artificial choice of vertex ordering.

\subsection{Flat connections and reduced $3$-webs}
Suppose $\G$ is drawn on a multiply-connected planar domain $\Sigma$, and $\Phi$ is a \emph{flat} $\SL_3$ connection.

A vertex in a $3$-multiweb is \emph{trivalent} if it has three adjacent edges of multiplicity $1$.
A $3$-multiweb consists in a collection of trivalent vertices connected in pairs by chains of edges with
multiplicities alternating between $1$ and $2$ along the chains.

We say a $3$-multiweb is \emph{reduced} or \emph{nonelliptic} if each component, considered as a planar graph by itself, has no contractible faces with
$0,2$ or $4$ trivalent vertices.  

There are certain \emph{skein} relations by which the trace of any $3$-multiweb can be written as a linear combination
of traces of reduced $3$-multiwebs (see Figure \ref{skeins}).
\begin{figure}
\begin{center}\includegraphics[width=.8in]{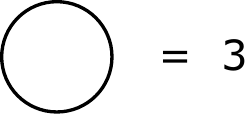}\\
\includegraphics[width=2.5in]{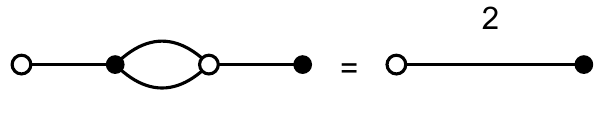}\\
\includegraphics[width=2.5in]{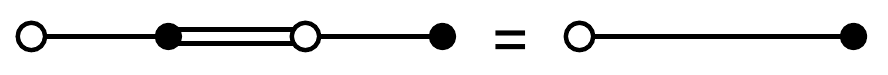}\\
\includegraphics[width=2.5in]{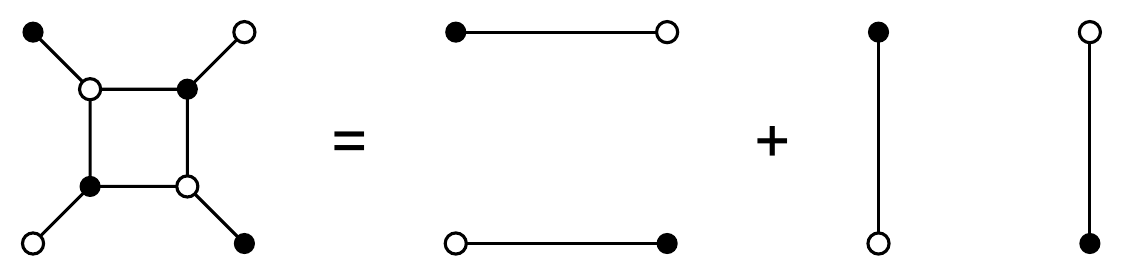}\end{center}
\caption{\label{skeins}Relations between $3$-multiwebs: a loop can be removed, increasing the weight 
of the remaining web by a factor of $3$; a bigon can be replaced with a single edge as shown multiplying
the weight by $2$; a doubled edge can be replaced with a single edge as shown, a square face can be resolved into a linear combination of two webs.}
\end{figure}
The set of isotopy classes of the reduced $3$-multiwebs which arise from reducing a given web
are well defined.

As a consequence for a flat $SL_3$-connection we can group the terms in the sum (\ref{SLnsum}) according to isotopy classes
of reduced $3$-multiwebs:

\be\label{SL3sumiso}\det\tilde K(\Phi)= \sum_{\lambda\in\Lambda_3} C_\lambda\Tr(\phi_\lambda).\ee

By a Theorem of Sikora and Westbury \cite{SW}, the functions $\Tr(\lambda)$ as $\lambda$ runs over $\Lambda_3$ 
form a linear basis for the regular (i.e. polynomial) functions on the character
variety of flat $SL_3$-local systems modulo gauge. 
As a consequence $C_\lambda$ can be in principle determined from $\det\tilde K(\Phi)$ as $\Phi$ varies over flat connections. 

\begin{problem} How can one extract $C_\lambda$ from $\det\tilde K(\Phi)$?
\end{problem}

Applying the skein relations to a non-reduced $3$-multiweb results in a collection of reduced webs
which depends on the order in which the skein relations are applied.
Even though the isotopy classes of the reduced web are well-defined, the individual webs themselves
will depend on the order. Is there a way to make a canonical choice in this reduction process,
so that starting from a random $3$-multiweb we arrive at a well-defined, canonical probability measure on reduced $3$-multiwebs?

\bibliographystyle{amsplain}

\end{document}